\documentclass[11pt]{amsart}

\usepackage{amsmath,amsthm,amsfonts,amssymb}

\usepackage{color}
\usepackage{amsaddr}

\newtheorem{thm}[equation]{Theorem}
\newtheorem{cor}[equation]{Corollary}
\newtheorem{lem}[equation]{Lemma}
\newtheorem{prop}[equation]{Proposition}

\theoremstyle{definition}

\numberwithin{equation}{section}

\definecolor{mjo}{rgb}{0,0,.9}

\newcommand{\cc}{\mathbb{C}}
\newcommand{\z}{\mathbb{Z}}

\newcommand{\s}{\mathfrak{sl}}

\newcommand{\B}{\mathfrak{b}}
\newcommand{\g}{\mathfrak{g}}

\newcommand{\Vir}{{\rm Vir}}

\newcommand{\Ind}{{\rm Ind}}
\newcommand{\Res}{{\rm Res}}

\definecolor{mjo}{rgb}{.4,0,.9}

\begin{document}



\title[]{The restriction of polynomial modules for the Virasoro algebra to $\mathfrak{sl}_2 ( \cc )$}

\author[Matthew Ondrus]{Matthew Ondrus}
\address{\noindent Mathematics Department,
Weber State University,
Ogden, UT  84408 USA,  \emph{E-mail address}: \tt{mattondrus@weber.edu}}

\author[Emilie Wiesner]{Emilie Wiesner*}
\address{\noindent  Department of Mathematics, Ithaca College, Ithaca, NY 14850, USA, \ \  \emph{E-mail address}: {\tt ewiesner@ithaca.edu} *Corresponding Author.}


\date{}

\begin{abstract}
The Lie algebra $\s_2 ( \cc )$ may be regarded in a natural way as a subalgebra of the infinite-dimensional Virasoro Lie algebra, so it is natural to consider connections between the representation theory of the two algebras.  In this paper, we explore the restriction to $\s_2 ( \cc )$ of certain induced modules for the Virasoro algebra.  Specifically, we consider Virasoro modules induced from so-called polynomial subalgebras, and we show that the restriction of these modules results in twisted versions of familiar modules such as Verma modules and Whittaker modules.  
\end{abstract}

\maketitle

\section{Introduction}

The Lie algebras $\s_2(\cc)$ and the Virasoro algebra are important exemplars in Lie theory. Both algebras have been well-studied and have a richly developed representation theory, including highest weight modules and Whittaker modules.

The Lie algebra $\s_2(\cc)$ naturally embeds in the Virasoro algebra, which suggests questions about how the representation theory of the two algebras are connected.   One might expect that a reasonable starting place would be with known families of modules that have been well-studied for both the Virasoro algebra and $\s_2 (\cc )$, such as highest weight modules and Whittaker modules. When restricted to $\s_2 (\cc )$, however, these $\Vir$-modules tend to be quite large and complicated. (See \cite{KR} and  \cite{ondWies2009}, and \cite{mazZhao2014}  for descriptions of highest-weight and Whittaker modules for $\Vir$, respectively.) 

Martin and Prieto \cite{marPrie} have given some general conditions for which weight $\s_2(\cc)$-modules can be lifted to $\Vir$-modules, by first relating the representation theory of $\s_2 ( \cc )$ to the representation of certain subalgebras ${\bf Witt}_>$ and ${\bf Witt}_<$; they also provide concrete connections between some families of $\s_2 (\cc)$- and $\Vir$-weight modules. 
Dong, Lin, and Mason \cite{DLM} have studied a piece of the reverse question, investigating the nature of vertex operator algebras (which are Virasoro modules) when restricted to $\s_2(\cc)$.   

In this paper, we focus on the restriction of ``polynomial modules," certain non-weight $\Vir$-modules, to $\s_2(\cc)$; these modules were introduced in \cite{ondWies18} as a generalization of the modules in $\Omega ( \lambda , b)$ studied in \cite{LZ14}.  The polynomial modules are induced from ``polynomial" subalgebras which have finite codimension in $\Vir$; and  in \cite{ondWies18}, it is shown that these subalgebras include all subalgebras of codimension one.
Here we prove that if $\s_2(\cc)$ is regarded as a subalgebra of $\Vir$ in a canonical way,  polynomial modules restrict to non-trivial twists (via $\s_2(\cc)$-automorphisms) of several classical families of $\s_2(\cc)$-modules, including Verma and Whittaker modules. 

These nontrivial twists of familiar $\s_2(\cc)$-modules can be induced from subalgebras of $\s_2(\cc)$; the subalgebras obviously have finite codimension, and subalgebras of a given codimension can be related via an $\s_2(\cc)$-automorphism.
Thus, the results of the paper can be alternatively summarized in the following way.  With the given canonical embedding of $\s_2(\cc)$ in $\Vir$, we show that most $\s_2(\cc)$-modules induced from nontrivial, proper subalgebras -- specifically those that are not subalgebras of $\B^+$, or $\B^-$ (see notation in Lemmas \ref{lem:2dimSubs} and \ref{lem:onedim}) -- can be lifted to a $\Vir$-module induced from a subalgebra of the same codimension.   If the $\s_2(\cc)$-module is induced from a subalgebra of codimension $1$, the lift is unique.  If the $\s_2(\cc)$-module is induced from a subalgebra of codimension $2$, there is an infinite family of lifts indexed by $\cc$.

Section \ref{sec:sl2} includes a variety of results on induced modules for $\s_2(\cc)$. While these results are generally known, we provide some non-standard perspectives on these modules that lead to connections to $\Vir$-modules. (These connections are the ultimate focus of the paper.)  In particular, we describe the one and two-dimensional subalgebras of $\s_2(\cc)$ as automorphism-twists of the standard Cartan, positive nilpotent, and Borel subalgebras. We then review families of modules induced from these subalgebras, including twists of Whittaker modules, Verma modules, and a family of modules $X( \xi )$ induced from the Cartan subalgebra.  In order to interpret later results (in Section \ref{sec:connectToVir}) in a more concrete way, we present some elementary facts about the structure of $X ( \xi )$ in Section \ref{sec:Xxi}.  For instance, we show that the modules $X ( \xi )$ have the ``dense'' weight modules found in \cite[p. 69]{Maz-sl2} as simple subquotients.  

The main results of the paper are found in Section \ref{sec:connectToVir}. They build on the fact that $\s_2(\cc)$ embeds in $\Vir$, and thus $\Vir$-modules may be restricted to $\s_2(\cc) $-modules.  We first review the results from \cite{ondWies18} on polynomial modules, which are $\Vir$-modules induced from ``polynomial" subalgebras.  These modules are distinct from highest weight modules or Whittaker modules for $\Vir$ and suggest no obvious connection to highest weight modules or Whittaker modules for $\s_2(\cc)$.  However, we prove that the polynomial modules of $\Vir$ restrict to the twists of the familiar $\s_2(\cc)$-modules described in Sections \ref{sec:1d} and \ref{sec:2d}.  As an application of these connections between $\s_2 (\cc)$ and $\Vir$, we prove Corollary \ref{cor:tensorVermas}, which describes the tensor product of two twisted $\s_2 ( \cc )$-Verma modules in terms of modules presented in Section \ref{sec:1d}.   In Section \ref{sec:UniqueLift}, we reframe the earlier results of the paper by considering which $\s_2(\cc)$-modules lift to $\Vir$-modules and how many such lifts there are.  Answers to these questions are summarized in Theorem \ref{thm:almostUnique}.

\section{$\s_2 ( \cc )$ and its modules} \label{sec:sl2}

In this section, we gather together a variety of results on induced modules for $\s_2 (\cc)$, formulated with our results on the Virasoro algebra in mind. This formulation focuses on twists, by $\s_2 (\cc)$-automorphisms, of ``standard" induced modules.  Following the usual notation, we fix a basis  $\{ e, h, f \}$ for $\s_2$ (short for $\s_2 ( \cc )$), with the usual relations $[h, e] = 2e$, $[e,f] = h$, and $[h,f] = -2f$.

The automorphisms of finite-dimensional Lie algebras have been carefully described. (See \cite[Chapter IX]{J}.) Here, we identify three automorphisms that will be of particular use in connecting subalgebras of $\s_2$ and $\Vir$.

For $\lambda \in \cc$, define  
\begin{equation} \label{eqn:gammalambda}
\gamma_\lambda (e) = e - \lambda h - \lambda^2 f, \quad \gamma_\lambda (h) = h + 2 \lambda f, \quad \gamma_\lambda ( f) = f.
\end{equation}
For $\lambda_1, \lambda_2 \in \cc$ with $\lambda_1 \neq \lambda_2$, define 
\begin{align} \label{eqn:tau}
\gamma_{\lambda_1, \lambda_2}(e) &= \frac{e - \lambda_1 h - \lambda_1^2 f}{\lambda_2 - \lambda_1},\\
 \gamma_{\lambda_1, \lambda_2} (h) &=  \frac{-2e + ( \lambda_1 + \lambda_2) h + 2 \lambda_1 \lambda_2 f}{\lambda_2 - \lambda_1}, \nonumber \\ \gamma_{\lambda_1, \lambda_2} ( f) &= \frac{- e + \lambda_2 h + \lambda_2^2 f}{\lambda_2 - \lambda_1}. \nonumber
\end{align}
Finally,  define
\begin{equation} \label{eqn:sigma}
\sigma (e) = f, \quad \sigma (f) = e, \quad \sigma (h) = -h.
\end{equation}

\subsection{One-dimensional subalgebras of $\s_2$ and induced modules} \label{sec:1d}

We now describe the modules that can be induced from one-dimensional modules for subalgebras of $\s_2$. We first construct two ``standard" induced modules and then describe all one-dimensional subalgebras of $\s_2$. Finally, we describe all modules induced from one-dimensional modules for subalgebras of $\s_2$ in relation to the standard induced modules. 

For $\eta \in \cc$, let $\cc_\eta$ be the one-dimensional $\cc e$-module given by $e1 = \eta$.  We form the $\s_2$-module $W( \eta )$ by inducing $\cc_\eta$ to an $\s_2$-module.  Specifically, define 
$$W( \eta ) = U ( \s_2 ) \otimes_{U( \cc e)} \cc_\eta.$$

If $\eta \neq 0$, then $W( \eta )$ is the \emph{standard Whittaker module}.  (See \cite{AP} for more background.) The structure of the standard Whittaker module $W(\eta)$ is well-known, even in the more general setting of complex finite-dimensional simple Lie algebras (e.g. see \cite{Kost78}).  For instance, there is a one-to-one correspondence between submodules of $W( \eta )$ and ideals of the center $\cc [c]$ of $U ( \s_2 )$, where 
$$c = 4ef + (h-1)^2 = 4fe + (h+1)^2 \in U ( \s_2 )$$ 
is the Casimir element.   In particular, if $\xi \in \cc$, then the quotient 
\begin{equation}\label{eqn:WhittSimpleQuotient}
W( \eta ) / ( c - \xi ) W( \eta )
\end{equation}
is a simple $\s_2$-module.  (See Section 3 of \cite{Kost78}.)

For $\xi \in \cc$, let $\cc_\xi$ be the one-dimensional $\cc h$-module with $h1 = \xi$.   Let 
\begin{equation}\label{eqn:defX_mu}
X(\xi) = \Ind_{\cc h}^{\s_2} ( \cc_\xi ) = U ( \s_2 ) \otimes_{U( \cc h)} \cc_\xi.
\end{equation}

The modules $X ( \xi )$ play an important role in Section \ref{sec:connectToVir} in Proposition \ref{prop:Res2} and Corollary \ref{cor:tensorVermas}, and thus we take time to understand in detail the structure of $X ( \xi )$.  Lemma \ref{lem:Xfiltration} and Proposition \ref{prop:quotientXmu} demonstrate a relationship between $X(\xi)$ and the center, $\cc[c]$, of $U(\s_2 )$, that parallels standard Whittaker modules: that is, there is a mapping between submodules of $X(\xi)$ and ideals of $\cc [c]$.  Then, in Proposition \ref{prop:twist}, we show that every $\s_2$-module induced from a one-dimensional module for a one-dimensional subalgebra of $\s_2$ is a twist of either $X(\xi)$ or $W(\eta)$.

\subsubsection{The $\s_2$-module $X(\xi)$} \label{sec:Xxi}

Let $x_\xi= 1 \otimes 1 \in X(\xi)$, so that $hx_\xi= \xi x_\xi$.  By the construction of $X(\xi)$, if $V$ is an $\s_2$-module generated by an $h$-eigenvector $v$ of weight $\xi$, then there is a surjective map $X(\xi) \to V$ with 
$$u x_\xi \mapsto u v$$
for all $u \in U ( \s_2)$.  We use a method similar to that of (\ref{eqn:WhittSimpleQuotient}) in order to obtain subquotients of $X( \xi )$ that are familiar.  Specifically, we use ideals of the center of $U(\s_2)$ generated by elements of the form $( c - \tau )^n$, where $\tau \in \cc$.

\begin{lem} \label{lem:Xfiltration}
Define $X(\xi)$ as in (\ref{eqn:defX_mu}), and let $\tau \in \cc$.  Then the following hold.  
\begin{enumerate}
\item[(i)] 
For all $0 \neq v \in X(\xi)$, $(c - \tau ) v \neq 0$.

\item[(ii)] $X(\xi) \supsetneq (c-\tau) X(\xi) \supsetneq (c-\tau)^2 X(\xi) \supsetneq \cdots$. 
\item[(iii)] $\displaystyle \bigcap_{n=1}^\infty (c-\tau)^n X(\xi)=0$.
\end{enumerate}
\end{lem}

\begin{proof}
Using a standard PBW basis for $U(\s_2)$, we see that $X(\xi)$ has a basis $\{ f^k e^l {x}_\xi \mid k, l \in \z_{\geq 0} \}$.  

The desired results are then straightforward consequences of the fact that 
\begin{align}
(c-\tau) f^k e^l {x}_\xi &= f^k e^l (c-\tau) {x}_\xi \nonumber \\
&= f^k e^l (4fe) {x}_\xi + ((\xi+1)^2 - \tau) f^k e^l  {x}_\xi\nonumber \\
&= 4 f^{k+1} e^{k+1} {x}_\xi + \sigma \label{eqn:fkel}
\end{align}
where $\sigma \in \mbox{span} \{ f^i e^j x_\xi \mid i \leq k, j \leq l \}$. 
\end{proof}

\medskip 
For $n \ge 0$, define 
\begin{equation} \label{eqn:defX_mubar}
\overline X(\xi, \tau)^n =(c-\tau)^n X(\xi) / ( c - \tau )^{n+1} X(\xi).
\end{equation}
As a shorthand, we will write $\overline X(\xi, \tau):= \overline X(\xi, \tau)^0$.

We begin with several basic facts about the modules $\overline X ( \xi , \tau )^n$.

\begin{lem}\label{lem:Xmu-subQuotSame}
Let $\xi, \tau \in \cc$.  Then $\overline X(\xi, \tau)^n  \cong \overline X(\xi, \tau)^{n+1}$ for every $n \in \z_{\ge 0}$. In particular, $\overline X(\xi, \tau)^n \cong \overline X(\xi, \tau)$ for all $n  \in \z_{\ge 0}$.
\end{lem}
\begin{proof}
For a fixed $n \ge 0$, define a map $\varphi : \overline X(\xi, \tau)^n  \rightarrow \overline X(\xi, \tau)^{n+1}$ by 
$$v + ( c - \tau )^{n+1} X(\xi) \longmapsto (c - \tau ) v + ( c - \tau )^{n+2} X(\xi).$$
It is easy to see that $\varphi$ is well defined and surjective.  Lemma \ref{lem:Xfiltration} can be used to show that $\varphi$ is injective.
\end{proof}

\begin{lem} \label{lem:spanHmu}
Let $\overline x_\xi = x_\xi + ( c - \tau ) X(\xi) \in \overline X(\xi, \tau)$.  The set $\{e^i \overline{x}_\xi \mid i \in \z_{\geq 0} \} \cup \{f^j \overline{x}_\xi \mid i \in \z_{> 0} \}$ is a basis for the module $\overline X(\xi, \tau)$ in (\ref{eqn:defX_mubar}).
\end{lem}
\begin{proof}
Equation (\ref{eqn:fkel}) can be used to show that the set spans $\overline X(\xi, \tau)$.   To show this set is linearly independent, note that 
\begin{align*}
he^l \overline{x}_\xi &= e^l (h+2l) \overline{x}_\xi\\
&= (\xi+2l) e^l \overline{x}_\xi;\\
h f^k \overline{x}_\xi&= (\xi-2k) \overline{x}_\xi.
\end{align*}
Since each vector has a different weight with respect to the action of $h$, the set must be linearly independent.
\end{proof}

\medskip
In anticipation of Theorem \ref{thm:dense}, we now review the $\s_2$-modules ${\bf V}( \overline{\xi}, \tau )$ described in \cite{Maz-sl2}.  For $\xi \in \cc$, regard $\overline{\xi}=\xi + 2 \z$ as a coset for the additive subgroup $2 \z \subseteq \cc$.  We define ${\bf V}( \overline{\xi} , \tau)$ to be ${\rm span}_\cc \{ v_\eta \mid \eta \in \xi + 2 \z \}$ with $\mathfrak{sl}_2$-action given by 
$$f v_\eta = v_{\eta -2}, \quad h v_\eta = \eta v_\eta, \quad e v_\eta = \frac 14 ( \tau - ( \eta + 1)^2) v_{\eta + 2}$$
for all $\eta \in \xi + 2 \z$.  The irreducibility of ${\bf V}( \overline{\xi}, \tau )$ depends upon whether $\xi + 2 \z$ contains any root of the polynomial $g_\tau (x) = \tau - ( x+1)^2$.  

\begin{thm}[\cite{Maz-sl2}, Theorem 3.29]\label{thm:irredWeightMods}
Let $\xi, \tau \in \cc$. The $\mathfrak{sl}_2$-module ${\bf V}( \overline{\xi}, \tau )$ is irreducible if and only if $\tau - ( \xi + 2i+1)^2 \neq 0$ for all $i \in \z$.
\end{thm}

A similar, but slightly relaxed, condition can be used to determine whether ${\bf V}( \overline \xi , \tau )$ is generated by the vector $v_{\xi'}$ for any $\xi' \in \overline \xi$.

\begin{lem}
Let $\xi, \tau \in \cc$, and $\xi' \in \overline \xi$.  Then ${\bf V}( \overline \xi , \tau )$ is generated by $v_{\xi'}$ if and only if $\tau - ( \xi' + 2i+1)^2 \neq 0$ for all $i \ge 0$. 
\end{lem}
\begin{proof}
The assumption that $\tau - ( \xi' + 2i+1)^2 \neq 0$ for all $i \ge 0$ ensures that $U ( \mathfrak{sl}_2 ) v_{\xi'}$ contains ${\rm span}_\cc \{ v_{\xi' + 2j} \mid j \ge 0 \}$, and the formula for the action of $f$ then forces $U ( \mathfrak{sl}_2 ) v_{\xi'} = {\bf V}( \overline \xi , \tau )$.  Alternatively, if $\tau - ( \xi' + 2i_0+1)^2 = 0$ for some $i_0 \ge 0$, then the formulas for the action of $e$, $f$, and $h$ imply that ${\rm span}_\cc \{ v_{\xi' + 2j} \mid j \le i_0 \}$ is a proper submodule of ${\bf V}( \overline \xi , \tau )$ containing $v_{\xi'}$.  Thus, $v_{\xi'}$ does not generate ${\bf V}( \overline \xi , \tau )$ in this case.
\end{proof}

Recall from Lemma \ref{lem:Xmu-subQuotSame} that $\overline X ( \xi, \tau )^n \cong \overline X ( \xi , \tau )$ for all $n \in \z_{\ge 0}$.

\begin{prop}\label{prop:quotientXmu}
Let $\xi, \tau \in \cc$, and write $\overline{\xi}=\xi + 2\z$.  Then $\overline X(\xi, \tau) \cong {\bf V}(\overline{\xi}, \tau)$ if and only if $\tau-(\xi+2j+1)^2  \neq 0$ for every $j \in \z_{\ge 0}$.
\end{prop}

\begin{proof}
First suppose $ \tau-(\xi+2j+1)^2  \neq 0$ for all $j \in \z_{\ge 0}$. Define a map $\overline X(\xi, \tau) \rightarrow {\bf V}(\overline{\xi}, \tau)$ by 
\begin{align*}
\overline{x}_\xi & \mapsto v_\xi;\\
f^k \overline{x}_\xi &\mapsto v_{\xi-2k}, \quad k \in \z_{>0}; \\
e^l \overline{x}_\xi & \mapsto \left( \prod_{j=0}^{l-1} \frac14 (\tau-(\xi+2j+1)^2) \right) v_{\xi+2l}.
\end{align*}
It follows from Lemma \ref{lem:spanHmu} that this map is a vector-space isomorphism.  It is straightforward to check that the map preserves the $\s_2$-action, so $\overline X(\xi, \tau) \cong {\bf V}(\overline{\xi}, \tau)$.

Conversely, suppose $ \tau-(\xi+2j+1)^2  = 0$ for some $j \in \z_{\ge 0}$. Any homomorphism $\theta: \overline X(\xi, \tau) \rightarrow {\bf V}(\overline{\xi}, \tau)$ must preserve weight spaces, so that $\theta(e^j \bar{x}_\xi) = c_j v_{\xi+2 j}$ for some $c_j \in \cc$.   Then $\theta (e^{j+1} \bar{x}_\xi) = e \theta(e^j \bar{x}_\xi) = e ( c_j v_{\xi+2j})=0$.  Therefore, $\theta$ is not an isomorphism.
\end{proof}

Note that for each $\xi$ and $\tau$, the equation $\tau-( \xi +2j+1)^2  = 0$ is satisfied by at most two integers $j$.  If $\tau-(\xi +2j+1)^2 \neq 0$ for all $j \in \z_{\ge 0}$, then ${\bf V}(\overline{\xi}, \tau) \cong \overline X(\xi , \tau)$ by Proposition \ref{prop:quotientXmu}.  Alternatively, if $\tau-(\xi +2j+1)^2 = 0$ for some $j \in \z_{\ge 0}$, choose $j_0 \in \z_{\ge 0}$ maximal with $\tau-(\xi +2j_0 +1)^2 = 0$.   Take $m \in \z$ with $m > j_0$, and let $\xi' = \xi + 2m$.   Then ${\bf V}(\overline{\xi'}, \tau) \cong {\bf V}(\overline{\xi}, \tau)$ since $\overline{\xi'} = \overline{\xi}$, and $\tau - ( \xi' + 2j + 1)^2 = \tau - ( \xi + 2 (m + j) + 1)^2 \neq 0$ for all $j \ge 0$ since $m + j > j_0 + j$.  Thus $\overline X(\xi', \tau) \cong {\bf V}(\overline{\xi'}, \tau) \cong {\bf V}(\overline{\xi}, \tau)$.  Thus we see that every module ${\bf V}(\overline{\xi}, \tau)$ has the form $\overline X(\xi', \tau)$ for some $\xi' \in \cc$.

\bigskip 
In the following lemma, we write $M(\delta)$ to represent the highest weight Verma module of highest weight $\delta \in \cc$; and $V(\delta)$ to represent the lowest weight Verma module of lowest weight $\delta$.  To understand the structure of $\overline X(\xi, \tau)$, Proposition \ref{prop:quotientXmu} implies that it remains only to consider the case when there exists $j \in \z_{\ge 0}$ such that $\tau-(\xi+2j+1)^2 = 0$.

\begin{lem} \label{lem:Xcompseries}
Let $\xi, \tau \in \cc$, and suppose there is $j_0 \in \z_{\ge 0}$ minimal with $\tau = ( \xi + 2j_0 +1)^2$.  

Then 
$Y(\xi, \tau) ={\rm span}_\cc \{ e^{j_0 + i} \overline x_\xi  \mid i \in \z, i > 0 \}$ is an $\mathfrak{sl}_2$-submodule of $\overline X(\xi, \tau)$ such that 
$$\overline X(\xi, \tau) /Y(\xi, \tau) \cong M( \xi + 2 j_0 ); \quad  Y(\xi, \tau) \cong V(\xi + 2 j_0 + 2).$$
\end{lem}
\begin{proof}
It is clear that $Y(\xi, \tau)$ is invariant under the action of $e$ and $h$.  We now consider the action of $f$ on this space.  Note that 
\begin{align*}
f e^{j_0 +1} \overline x_\xi&=\frac 14 (fe) e^{j_0} \overline x_\xi = \frac 14 \left( c - (h+1)^2 \right) e^{j_0} \overline x_\xi \\
&= \frac 14 \left( c - (\xi + 2j_0 +1)^2 \right) e^{j_0} \overline x_\xi = \frac 14 ( c - \tau ) e^{j_0} \overline x_\xi = 0.
\end{align*}
Now consider  $f e^{j_0 +i} \overline x_\xi$ for $i > 1$:
$$f e^{j_0 +i} \overline x_\xi= fe e^{j_0 + i-1} \overline x_\xi= \frac 14 (c - (h+1)^2) e^{j_0 + i-1} \overline x_\xi \in \cc e^{j_0 + i-1} \overline x_\xi$$
since both $c$ and $h$ act by a scalar on $e^{j_0 + i-1} \overline x_\xi$. Therefore, $Y(\xi, \tau)$ is invariant under the action of $e$ and thus a submodule.

To prove the isomorphism, let $\overline x$ represent the image of $x \in \overline X( \xi, \tau)$ in $\overline X(\xi, \tau) / Y(\xi, \tau)$. Note that $\overline{e^{j_0} \overline x_\xi}$ is a highest weight vector in $\overline X(\xi, \tau)/ Y(\xi, \tau)$ that generates the module. Thus there is a surjective map from $M( \xi + 2 j_0 )$ to $\overline X(\xi, \tau)/ Y(\xi, \tau)$. Since Lemma \ref{lem:spanHmu} implies that $\{ \overline{e^{l} \overline x_\xi}, \overline{f^{k} \overline x_\xi}\mid 0 \le l \leq j_0, k >0 \}$ is a basis for $\overline X(\xi, \tau) / Y(\xi, \tau)$, this map must be an isomorphism.

To complete the proof, we see that $Y ( \xi , \tau )$ is generated by the vector $e^{j_0 + 1} \overline x_\xi$.  Moreover, $f ( e^{j_0 + 1} \overline x_\xi ) = 0$, and $h ( e^{j_0 + 1} \overline x_\xi ) = ( \xi + 2 j_0 + 2 ) e^{j_0 + 1} \overline x_\xi$.  That is, $Y(\xi, \tau)$ is a lowest weight Verma module of lowest weight $\xi + 2 j_0 + 2$
\end{proof}

Note that the Verma modules $M( \xi + 2 j_0 )$ and $V(\xi + 2 j_0 + 2)$ are either simple or have a unique simple quotient. Thus, Lemma \ref{lem:Xcompseries} effectively gives a composition series for $\overline X(\xi, \tau)$ if $\tau-(\xi+2j_0 +1)^2  = 0$. As proved in \cite{Maz-sl2}, this is the same composition series as for ${\bf V} (\overline \xi, \tau)$ even though Proposition \ref{prop:quotientXmu} shows that $\overline X(\xi, \tau)^n \not\cong {\bf V}(\overline{\xi}, \tau)$ when $\tau - ( \xi + 2j_0 + 1)^2 = 0$ for some $j_0 \ge 0$.

We summarize the results of this section in the following theorem.
\begin{thm} \label{thm:dense}
For every $\tau \in \cc$, the $\s_2$-module $X ( \xi )$ has a filtration 
$$X(\xi) \supseteq (c-\tau) X(\xi) \supseteq (c-\tau)^2 X(\xi) \supseteq \cdots$$
with corresponding subquotients isomorphic to $\overline X ( \xi , \tau ) = X ( \xi ) / ( c - \tau ) X ( \xi )$.  
\begin{itemize}
\item[(i)] If $\tau - ( \xi + 2 j + 1 )^2 \neq 0$ for all $j \in \z_{\ge 0}$, then $\overline X(\xi, \tau) \cong {\bf V}( \overline{\xi}, \tau )$, where ${\bf V}( \overline{\xi}, \tau )$ is the module described in \cite{Maz-sl2}.  

\item[(ii)] If $\tau - ( \xi + 2 j_0 + 1 )^2 = 0$ for some minimal $j_0 \in \z_{\ge 0}$, then there is a proper submodule $Y(\xi, \tau) \subseteq \overline X ( \xi , \tau )$ such that 
 $$\overline X(\xi, \tau) /Y(\xi, \tau) \cong M( \xi + 2 j_0 ); \quad Y(\xi, \tau) \cong V(\xi + 2 j_0 + 2).$$ 
 \end{itemize}
\end{thm}

\subsubsection{Modules induced from one-dimensional subalgebras of $\s_2$}

First we state a general lemma that allows us to characterize modules as ``twists" of other modules. Then we describe all one-dimensional subalgebras of $\s_2$ and use this to relate all modules induced from one-dimensional modules of these subalgebras as twists of known induced modules.

Suppose $\g, \g'$ are Lie algebras and $\gamma: \mathfrak g \rightarrow \mathfrak g'$ is an isomorphism. For a $\mathfrak g'$-module $M$, define  the \textit{$\gamma$-twist} of $M$, denoted $M^\gamma$,  to be  the $\g$-module with underlying space $M$ and $\g$-action defined via the map $\mathfrak g \to \mathfrak g' \to {\rm End} (M)$.  In particular, if $M$ is a $\gamma ( \mathfrak b )$-module for some subalgebra $\mathfrak b \subseteq \g$ and some automorphism $\gamma : \g \to \g$, then $M^\gamma$ is a $\mathfrak b$-module.

\begin{lem}\label{lem:inducedtwist}
Let $\g$ be a Lie algebra with subalgebra $\mathfrak b$, and let $\gamma : \g \to \g$ be a Lie algebra automorphism.  Suppose $M$ is a $\gamma (\mathfrak b)$-module.  Then 
$$\Ind_{\gamma(\mathfrak b)}^\g (M) \cong \left(  \Ind_{ \mathfrak b }^\g \left( M^{\gamma} \right) \right)^{\gamma^{-1}}.$$
\end{lem}
\begin{proof}
Since $\B$ acts on $\Ind_{ \mathfrak b }^\g \left( M^{\gamma} \right)$, it follows that $\gamma ( \B )$ acts on $\left(  \Ind_{ \mathfrak b }^\g \left( M^{\gamma} \right) \right)^{\gamma^{-1}}$.  In particular, if $y \in \B$, then $\gamma (y)$ acts on $1 \otimes m \in \left(  \Ind_{ \mathfrak b }^\g \left( M^{\gamma} \right) \right)^{\gamma^{-1}}$ according to $\gamma (y) . ( 1 \otimes m ) = \gamma^{-1}(\gamma(y)) \otimes m=y \otimes m=1 \otimes y . m$.  Since we are regarding $m \in M^\gamma$, it follows that $y.m = \gamma (y) m$, and therefore 
$$\gamma (y) . ( 1 \otimes m ) = 1 \otimes \gamma (y) m.$$
Thus, there is a $\gamma(\B)$-homomorphism $M \rightarrow \Res_{\gamma(\B)}^{\g} \left(\Ind_{\B}^{\g} (M^{\gamma})\right)^{\gamma^{-1}}$ given by $m \mapsto 1 \otimes m$.  Frobenius reciprocity implies that 
$${\rm Hom}_{\gamma ( \B )} ( M , {\rm Res}_{\gamma ( \B )}^{\g} (N) ) \cong {\rm Hom}_{\g} ( {\rm Ind}_{\gamma ( \B )}^{\g} (M) , N ),$$
where $N = \left(  \Ind_{ \mathfrak b }^\g \left( M^{\gamma} \right) \right)^{\gamma^{-1}}$.  Hence the map $m \mapsto 1 \otimes m$ in ${\rm Hom}_{\gamma ( \B )} ( M , {\rm Res}_{\gamma ( \B )}^{\g} (N) )$ lifts to a map $\theta \in {\rm Hom}_{\g} ( {\rm Ind}_{\gamma ( \B )}^{\g} (M) , N )$ satisfying $\theta ( 1 \otimes m) = 1 \otimes m$.  Moreover, since the elements $1 \otimes m$, for $m \in M$, generate $\left( \Ind_{\B}^{\g} (M^{\gamma}) \right)^{\gamma^{-1}}$, this homomorphism is surjective.

Using similar reasoning, we can construct a surjective homomorphism $\Ind_{\B}^{\g} (M^{\gamma}) \rightarrow \left( \Ind_{\gamma(\B)}^{\g} (M) \right)^\gamma$, which gives a surjective homomorphism $\theta': \Ind_{\B}^{\g} (M^{\gamma})^{\gamma^{-1}} \rightarrow \left( \Ind_{\gamma(\B)}^{\g} (M) \right)$ such that (again) $\theta' (1 \otimes m) = 1 \otimes m$.  Since $\theta' (\theta (1 \otimes m )) = 1 \otimes m$ and the elements $1 \otimes m$ generate $\Ind_{\gamma(\B)}^{\g} (M)$, this implies that $\theta' \circ \theta = {\rm Id}$. Similarly, we can argue $\theta \circ \theta'= {\rm Id}$. It follows that both $\theta$ and $\theta'$ are isomorphisms.
\end{proof}

Next we describe one-dimensional subalgebras of $\s_2$ using the automorphisms defined in (\ref{eqn:gammalambda})-(\ref{eqn:sigma}).  It is well-known that the Cartan subalgebras of a finite-dimensional semisimple Lie algebra $\g$ are conjugate to each other via $\g$-automorphisms. (See \cite{H}.)  Thus, the subalgebras described in Lemma \ref{lem:onedim} (iii), (iv) are simply Cartan subalgebras of $\s_2$.  (In fact, they are all of the Cartan subalgebras.) The important content of this lemma -- for us -- is to describe the particular automorphisms that relate the various one-dimensional subalgebras of $\s_2$.

\begin{lem} \label{lem:onedim}
The one-dimensional subalgebras of $\s_2$ are precisely
\begin{itemize}
\item[(i)] $\mathfrak n_\lambda:=\gamma_\lambda(\cc e)$, $\lambda \in \cc$; 
 \item[(ii)] $\mathfrak n^-=\cc f=\sigma(\cc e)$;
\item[(iii)] $\mathfrak h_\lambda:=\gamma_\lambda(\cc h)$, $\lambda \in \cc$;  
\item[(iv)] $\mathfrak h_{\lambda_1, \lambda_2}:=\gamma_{\lambda_1, \lambda_2} (\cc h)$, $\lambda_1, \lambda_2 \in \cc$ with $\lambda_1 \neq \lambda_2$. 
\end{itemize}
\end{lem}

\begin{proof}
Let $\mathfrak a$ be a one-dimensonal subalgebra of $\s_2$. First suppose that elements of $\mathfrak a$ have a nonzero coefficient for $e$, so that $\mathfrak a$ has a basis of the form $e-\beta h - \delta f$. If $\delta= \beta^2$, then $\mathfrak a= \gamma_\beta (\cc e)$ since $\gamma_\beta (e) = e - \beta h - \beta^2 f$.  
If $\beta^2 \neq \delta$, then 
$\mathfrak a= \gamma_{\lambda_1, \lambda_2}(\cc h)$, where $\lambda_1=\beta+ \sqrt{\beta^2-\delta}, \lambda_2=\beta- \sqrt{\beta^2-\delta}$.

If the coefficient of $e$ is zero, then $\mathfrak a$ has a basis $h - \delta f$ and $\mathfrak a= \gamma_{-\frac{\delta}{2}}(\cc h)$; or $\mathfrak a$ has a basis $f$.
\end{proof}

Note that for any one-dimensional Lie algebra $\mathfrak a$, its simple modules have the form $\cc_\mu$, where $\mu: \mathfrak a \rightarrow \cc$ is a Lie algebra homomorphism and $\cc_\mu$ is defined by $x1=\mu(x)$.
The following proposition constructs any  $\s_2$-module induced from a simple module for a one-dimensional subalgebra of $\s_2$ as a twist of some $X(\nu)$ or $W(\nu)$.

\begin{prop} \label{prop:twist}
Let $\mathfrak a$ be a one-dimensional $\s_2$-subalgebra and $\mu: \mathfrak a \rightarrow \cc$ a Lie algebra homomorphism.  
\begin{itemize}
\item[(i)] If $\mathfrak a=\mathfrak h_{\lambda}$ or $\mathfrak h_{\lambda_1, \lambda_2}$, then 
$$\Ind_{\mathfrak a}^{\s_2} (\cc_\mu) \cong X(\mu(\gamma(h)))^{\gamma^{-1}}$$ 
where $\gamma=\gamma_\lambda$ or $\gamma_{\lambda_1, \lambda_2}$ respectively.  
\item[(ii)] If $\mathfrak a=\mathfrak n_{\lambda}$ or $\mathfrak n^-$, then 
$$
\Ind_{\mathfrak a}^{\s_2} (\cc_\mu) \cong W(\mu(\gamma(e)))^{\gamma^{-1}}
$$ where $\gamma=\gamma_\lambda$ or $\sigma$ respectively. 
\end{itemize}
\end{prop}
\begin{proof}
 Note that in $(\cc_\mu)^{\gamma}$, $d.1 = \gamma (d)1 = \mu ( \gamma (d))$ for $d$ belonging to $\mathfrak h$ or $\mathfrak n$ and for the $\gamma$ corresponding to $\mathfrak a$. Therefore, the result follows from Lemmas \ref{lem:inducedtwist} and \ref{lem:onedim}.
\end{proof}
If $\mu(\gamma(e)) \neq 0$, then $W( \mu(\gamma(e)) )^{\gamma^{-1}}$ is a twist of a Whittaker module.

\subsection{Two-dimensional subalgebras of $\mathfrak{sl}_2$ and induced modules} \label{sec:2d}

Now we consider two-dimensional subalgebras of $\s_2$.  Note that any two-dimensional Lie algebra is solvable and any two-dimensional subalgebra of $\s_2$ is maximal. Therefore, the two-dimensional subalgebras of $\s_2$ are precisely the Borel subalgebras of $\s_2$.  Similar to our note prior to Lemma \ref{lem:onedim}, it is well-known that the Borel subalgebras of a finite-dimensional Lie algebra are conjugate via $\g$-automorphisms.(Again, see \cite{H}.) Here, we identify specific automorphisms that allow a connection to polynomial subalgebras of $\Vir$.

\begin{lem}\label{lem:2dimSubs}
The two-dimensional subalgebras of $\s_2$ are exactly
\begin{itemize}
\item[(i)] $\B^+={\rm span}_\cc \{ h, e\}$;
\item[(ii)] $\B^-=\sigma(\B^+)={\rm span}_\cc \{ h, f\}$;
\item[(iii)] $\mathfrak{b}_\lambda:= \gamma_\lambda(\B^+)$ 
for $\lambda \in \cc^*$. 
\end{itemize}
\end{lem}

It is straightforward to check that $\gamma_{\lambda_1, \lambda_2}(\B^+) = \gamma_{\lambda_1}(\B^+)$ for $\lambda_1, \lambda_2 \in \cc^*$ with $\lambda_1 \neq \lambda_2$. Here, we choose to focus on the form $\gamma_{\lambda_1}(\B^+)$ because this form maps more clearly to a subalgebra of the Virasoro algebra, as described in Section \ref{sec:connectToVir}.

\begin{proof}
Note that (\ref{eqn:gammalambda}) yields
\begin{equation} \label{eq:gammalambda}
\gamma_\lambda(h) = h+2 \lambda f, \quad \gamma_\lambda(e+\lambda h)=e + \lambda^2 f.
\end{equation}
Thus,  $\mathfrak{b}_\lambda={\rm span}_\cc \{ h+2 \lambda f, e + \lambda^2 f\}$.

Let $\mathfrak{b}$ be a two-dimensional subalgebra of $\s_2$. Since $\mbox{span} \{e, f\}$ is not a subalgebra, it must be that $\mathfrak{b}$ contains at least one element with a nonzero coefficient of $h$.  It is then straightforward to show in this case that $\mathfrak{b}$ has a basis of the form $\{h+ \alpha f,  e+ \beta f\}$ or  $\{h+ \alpha e,  f+ \beta e\}$.  In the first case, we note that 
$$
[h+\alpha f,e+ \beta f] = -\alpha h+2e-2\beta f \in \mathfrak{b}. 
$$
This forces $\alpha^2=4 \beta$. If $\alpha=0$, this gives $\B^+$. For $\alpha \neq 0$, we see that $\mathfrak b = \mathfrak b_\lambda$, where $\lambda = \frac{\alpha}{2}$.  Considering a basis of the second type, we either have $\B^-$ or $\mathfrak{b}_\lambda$ where $\lambda=\frac{4}{\alpha}$.
\end{proof}

Note that $\B^+=\gamma_0 (\B^+)=\mathfrak b_0$.  Since all two-dimensional Lie algebras $\B$ are solvable, their simple modules are all one-dimensional. As before, we may  express them in the form $\cc_\mu$, where $\mu: \B \rightarrow \cc$ is a Lie algebra homomorphism and $\cc_\mu$ is defined by $x1=\mu(x)$ for all $x \in \B$.

\begin{prop} \label{prop:Vermaiso}
Let $\lambda \in \cc$ and $\mu: \B_\lambda \rightarrow \cc$ be a homomorphism, and assume $\cc_\mu$ is a $\B_\lambda$-module.  Then,
$$\Ind_{\B_\lambda}^{\s_2} ( \cc_\mu ) \cong M(\mu(\gamma_\lambda(h)))^{\gamma_\lambda^{-1}},$$
where $M(\mu(\gamma_\lambda(h)) )$ is the Verma module of highest weight $\mu(\gamma_\lambda(h))$.
\end{prop}
\begin{proof}
By Lemma \ref{lem:inducedtwist}, we know that $\Ind_{\B_\lambda}^{\s_2} ( \cc_\mu ) \cong \left(  \Ind_{\B^+}^{\s_2} (( \cc_\mu )^{\gamma_\lambda} ) \right)^{\gamma_\lambda^{-1}}$, so it suffices to show that $\Ind_{\B^+}^{\s_2} (( \cc_\mu )^{\gamma_\lambda}) \cong M( \mu(\gamma_\lambda(h) )$.  But this is clear as the $\B^+$-module structure on $( \cc_\mu )^{\gamma_\lambda}$ is given by $h.1 = \mu(\gamma_\lambda(h))$.  Moreover, $e \in [ \B^+, \B^+ ]$, and $[ \B^+, \B^+ ] .1 = 0$.
\end{proof}

In the context of studying $U(\s_{n+1})$ modules which are free rank-$1$ $U(\mathfrak h)$-modules, Nilsson \cite{Nilsson} defined a certain $\s_2$-module $F_{(a,1)}M_b$ and showed these modules are simple if and only if $2b \not\in \z_{\geq 0}$. One can construct an isomorphism $F_{(\lambda,1)}M_{\frac12 \mu(\gamma_\lambda(h)) \lambda^{-1}} \cong \Ind_{\B_\lambda}^{\s_2} ( \cc_\mu )$. Proposition \ref{prop:Vermaiso} then reproduces Nilsson's irreducibility result in the $\s_2$ case.

\section{Restrictions of modules for the Virasoro algebra} \label{sec:connectToVir}

The Virasoro algebra $\Vir$ is a Lie algebra over $\cc$ with a basis $\{ z, e_i \mid i \in \z\}$, such that $z$ is central and, for $i, j \in \z$, $[e_i, e_j] = (j-i)e_{i+j} + \delta_{j,-i} \frac{i^3-i}{12} z$.  There is a natural embedding of $\s_2$ into $\Vir$: 
\begin{equation}\label{eqn:sl2-embed}
h \mapsto 2 e_0, \qquad e \mapsto e_1, \qquad f \mapsto - e_{-1}.
\end{equation}
Let $\mathfrak{sl}_2^{\Vir}=\mbox{span} \{e_1, e_0, e_{-1} \}$ be the image of $\s_2$ under this embedding. By an abuse of notation, we will consider the automorphisms of $\s_2$ to be automorphisms of $\s_2^\Vir$ via this embedding.

In this section, we connect the induced $\s_2$-modules considered in the first part of this paper with $\Vir$-modules induced from certain ``polynomial" subalgebras of $\Vir$. A central idea in making this connection is that subalgbras of $\Vir$ of of codimension $1$, $2$, or $3$, can be mapped, via intersection, to subalgebras of $\s_2$.

\subsection{The polynomial subalgebras of $\Vir$ and their induced modules}

Recapping results from \cite{ondWies18}, we describe the construction $\Vir$-subalgebras associated with Laurent polynomials $f \in \cc [t^{\pm 1}]$; and corresponding induced $\Vir$-modules. 

For $i \in \z$, identify $e_i \in \Vir$ with $t^i \in \cc [t^{\pm 1}]$. Extending this to linear combinations, we identify any Laurent polynomial $f(t)\in \cc [t^{\pm 1}]$ with an element of $\Vir$ and show that the subspace $\Vir^{f(t)}= \mbox{span} \{ z, \ t^i f(t) \mid i \in \z\}$ is, in fact, a Lie subalgebra of $\Vir$.  Since $\Vir^{f(t)} = \Vir^{t^k f(t)}$ for all $k \in \z$, it is important, in order to avoid redundancy, that we always write $f(t)$ as an element of the polynomial ring $\cc [t]$ and assume $f(t)$ has a nonzero constant coefficient.  In particular, we must assume the roots of $f(t)$ are nonzero.  If, for example, $f(t)=t-\lambda$, $\lambda \in \cc^*$, then 
$$\Vir^{t- \lambda}= \mbox{span} \{ z, e_i-\lambda e_{i-1} \mid i \in \z \}.$$
Similarly, for $\lambda_1, \lambda_2 \in \cc^*$ and $f(t)=(t-\lambda_1)(t-\lambda_2) = t^2 - ( \lambda_1 + \lambda_2 )t + \lambda_1 \lambda_2$, 
$$
\Vir^{(t-\lambda_1)(t-\lambda_2)} = \mbox{span} \{ z, e_i-(\lambda_1+\lambda_2) e_{i-1} +\lambda_1\lambda_2 e_{i-2}\mid i \in \z  \}.
$$

Suppose $f(t)=\prod_{i=1}^k (t-\lambda_i)^{n_i} $, where $\lambda_1, \ldots, \lambda_k \in \cc^*$ are distinct and $n_1, \ldots, n_k \in \z_{\geq 1}$. We demonstrate in \cite{ondWies18} that a linear map $\mu: \Vir^{f(t)} \rightarrow \cc$ is a Lie algebra homomorphism precisely when both $\mu(z)=0$ and there exist polynomials $p_1, \ldots,p_k$ with $\deg(p_i) < n_i$ such that $\mu(t^j f(t)) =p_1(j)\lambda_1^j+ \cdots + p_k(j) \lambda_k^j$.  In particular, the one-dimensional $\Vir^{f(t)}$-modules have the form $\cc_\mu$ for such a homomorphism $\mu$. 

We then define the induced modules $V^{f(t)}_\mu= \Ind_{\Vir^{f(t)}}^{\Vir} (\cc_\mu)$ and study the simplicity of these modules. As part of this work, we show the modules $V^{f(t)}_\mu= \Ind_{\Vir^{f(t)}}^{\Vir} (\cc_\mu)$ are non-isomorphic for distinct choices of $f(t)$ and $\mu$ and are simple provided $\mu$ satisfies some nondegeneracy condtions. In particular, the modules $V^{t- \lambda}_\mu$ and $V^{(t-\lambda_1)(t-\lambda_2)}_\mu$ are simple for all nonzero $\mu$.

\subsection{Restriction}

In Lemma \ref{lem:ResIndCommute} below, we consider the situation of a Lie algebra $\g$ with subalgebras $\mathfrak a$ and $\mathfrak b$ and a $\mathfrak b$-module $V$ such that the induced module $\Ind_\B^\g (V)$ is generated over $U( \mathfrak a)$ by $V$.  The following lemmas are used to show that Lemma \ref{lem:ResIndCommute} can be applied in the context of $\Vir$ with respect to certain subalgebras. 

\begin{lem}\label{lem:subalgebraAction}
Let $\g$ be a Lie algebra with $\mathfrak a \subseteq \g$ a Lie subalgebra, and suppose $V$ is a $\g$-module.  Assume $v \in V$ satisfies $\g v \subseteq \mathfrak a v$.  Then $U(\g) v \subseteq U(\mathfrak a) v$.
\end{lem}
\begin{proof}
It is enough to prove that for any $x_1, \ldots, x_n \in \g$, then $(x_1 \cdots x_n) v =\sum_{i}  s_{i,1} s_{i,2} \cdots  s_{i,n} v$,
where $s_{i,j} \in \mathfrak a$ and the sum is over finitely many $i$.    We show this by induction on $n$, noting that the $n=1$ case is given. 

By assumption, $x_1 v= S_1 v$ for some $S_1 \in \mathfrak a$; and by the inductive hypothesis, $(x_2 \cdots x_n) v = S_2 v$ where $S_2 =  \sum_{i}  s_{i,1} s_{i,2} \cdots  s_{i,n-1} v$. Then 
$$(x_1 \cdots x_n) v = x_1 S_2v = (S_2 x_1 + [x_1, S_2])v = S_2 S_1v + [x_1, S_2] v.$$
In this expression, $S_2 S_1v$ has the desire form. Moreover, 
$$[x_1, S_2] =\sum_{i}  [x_1,  s_{i,1} s_{i,2} \cdots  s_{i,n-1}]= \sum_i y_{i,1} y_{i,2} \cdots  y_{i,n-1}$$ for some $y_{i,j} \in \g$. Thus the term $[x_1, S_2] v$ also has the desired form by the inductive hypothesis.
\end{proof}

\begin{cor}\label{cor:sl2-cyclic}
Suppose $f(t) \in \cc [t] \subseteq \cc [t^{\pm 1}]$ has the form $\sum_{i=0}^k a_i t^i$, where $a_0, a_k \neq 0$ and $k \in \{ 1, 2, 3 \}$.  Then $\Vir^{f(t)} + \s_2^\Vir = \Vir$.

Moreover, if $ \mu : \Vir^{f(t)} \to \cc$ is an algebra homomorphism and if $v_{\mu}$ is the canonical generator of $V_{\mu}^{f(t)}$, then $V_{\mu}^{f(t)} = U( \mathfrak{sl}_2^\Vir ) v_{ \mu}$.
\end{cor}
\begin{proof}

To show that $\Vir^{f(t)} + \s_2^\Vir = \Vir$, note that the set $\{ t^i ( a_0 + \cdots + a_k t^k ) \mid i \in \z \} \cup \{ t^{-1} , t^0, t^1 \}$ spans $\Vir$ since $k \le 3$.

For the second part of the claim, we observe that $\Vir^{f(t)} + \s_2^\Vir = \Vir$ implies $\Vir \, v_\mu \subseteq \mathfrak{sl}_2^\Vir v_\mu$.  Then we may apply Lemma \ref{lem:subalgebraAction} with $\g=\Vir$ and $\mathfrak a= \s_2$. 
\end{proof}

\begin{cor}\label{cor:codimIntersect}
Suppose $f(t) \in \cc [t] \subseteq \cc [t^{\pm 1}]$ has the form $\sum_{i=0}^k a_i t^i$, where $a_0, a_k \neq 0$ and $k \in \{ 1, 2, 3 \}$.  Then $\Vir / \Vir^{f(t)} \cong \s_2^\Vir / ( \s_2^\Vir \cap \Vir^{f(t)} )$ as vector spaces.  Consequently $\dim \left( \s_2^\Vir / ( \s_2^\Vir \cap \Vir^{f(t)} ) \right) = k$ and $\dim ( \s_2^\Vir \cap \Vir^{f(t)} ) = 3-k$.
\end{cor}
\begin{proof}
The fact that $\Vir / \Vir^{f(t)} \cong \s_2^\Vir / ( \s_2^\Vir \cap \Vir^{f(t)} )$ follows immediately from the second isomorphism theorem.
\end{proof}

\begin{lem} \label{lem:ResIndCommute}
Let $\g$ be a Lie algebra and $\mathfrak{a}, \B$ be subalgebras of $\g$. Let $V$ be a $\B$-module and suppose $\Ind_\B^\g (V)=U(\mathfrak{a}) \otimes_{U(\B)} V$ as vector spaces. Then
$$\Ind^{\mathfrak a}_{\mathfrak a \cap \B}  \Res^{\B}_{\mathfrak a \cap \B}  (V) \cong \Res_{\mathfrak a}^{\g}  \Ind^{\g}_{\B} (V)$$
as $\mathfrak{a}$-modules.
\end{lem}

\begin{proof}
Because $\Ind_\B^\g (V)=U(\mathfrak{a}) \otimes_{U(\B)} V$ as vector spaces, it follows that $\Res_{\mathfrak a}^{\g}  \Ind^{\g}_{\B} (V) = U(\mathfrak{a}) \otimes_{U( \B )} V$ as vector spaces. We also have $\Ind^{\mathfrak a}_{\mathfrak a \cap \B}  \Res^{\B}_{\mathfrak a \cap \B}  (V) = U(\mathfrak{a}) \otimes_{U( \mathfrak a \cap \B )} V$ as vector spaces.  The PBW theorem implies that as vector spaces  $U(\mathfrak{a}) \otimes_{U( \mathfrak a \cap \B )} V \cong U(\mathfrak{a}) \otimes_{U( \B )} V$ where $u \otimes v \mapsto u \otimes v$ for $u \in U(\mathfrak{a}) $ and $v \in V$.

To show this extends to an $\mathfrak{a}$-module isomorphism, note that because $\Ind_\B^\g (V)=U(\mathfrak{a}) \otimes_{U(\B)} V$ as vector spaces, then $\Res_{\mathfrak a}^{\g} \Ind^{\g}_{\B} (V) = U(\mathfrak{a}) \otimes_{U( \B )} V$ as $\mathfrak{a}$-modules. In particular, the $\mathfrak a$-action is given by $a ( u \otimes v) = (au) \otimes v$. Similarly, the action of $\mathfrak a$ on the $\mathfrak a$-module $\Ind^{\mathfrak a}_{\mathfrak a \cap \B}  \Res^{\B}_{\mathfrak a \cap \B}  (V)$ is also given by $a ( u \otimes v) = (au) \otimes v$. 
\end{proof}

\begin{prop}\label{prop:resVirMod-sl2}
Suppose $f(t) \in \cc [t] \subseteq \cc [t^{\pm 1}]$ has the form $\sum_{i=0}^k a_i t^i$, where $a_0, a_k \neq 0$ and $k \in \{ 1, 2, 3 \}$, and let $\mu : \Vir^{f(t)} \to \cc$ be an algebra homomorphism.  Then 
$$\Res^\Vir_{\s_2^\Vir} ( V_{\mu}^{f(t)} ) \cong \Ind^{\s_2^\Vir}_{\s_2^\Vir \cap \Vir^{f(t)}} ( \cc_{\mu} ),$$
where we regard $\cc_{\mu}$ as a $\s_2^\Vir \cap \Vir^{f(t)}$-module (via restriction). 
\end{prop}
\begin{proof}
From Corollary \ref{cor:sl2-cyclic} and Lemma \ref{lem:ResIndCommute}, we have that 
\begin{align*}
\Res^\Vir_{\s_2^\Vir} ( V_{\mu}^{f(t)} ) &=  \Res^\Vir_{\s_2^\Vir} \Ind^\Vir_{\Vir^{f(t)}} ( \cc_{\mu} ) \\
& \cong \Ind^{\s_2^\Vir}_{\s_2^\Vir \cap \Vir^{f(t)}}  \Res^{\Vir^{f(t)}}_{\s_2^\Vir \cap \Vir^{f(t)}}  ( \cc_{\mu} ) \\
&\cong \Ind^{\s_2^\Vir}_{\s_2^\Vir \cap \Vir^{f(t)}} ( \cc_{\mu} ).
\end{align*}
\end{proof}

\subsection{Connecting $\s_2$-Verma modules and $\Vir$-modules $V^{t- \lambda}_\mu$}

In this section, we focus on connections between Verma modules for $\s_2$ and the $\Vir$-modules $V^{t- \lambda}_{\mu}$. In particular,  Proposition \ref{prop:Res1} shows that each nontrivial twist of an $\s_2$-Verma module lifts to a $\Vir$ module induced from a codimension one $\Vir$-subalgebra.

\begin{prop} \label{prop:Res1}
Let $\lambda \in \cc^*$ and  $\mu : \Vir^{t - \lambda} \to \cc$ a Lie algebra homomorphism. Then,
$$\Res_{\mathfrak{sl}_2^\Vir}^\Vir  (V^{t- \lambda}_{\mu}) \cong M\left(2 \mu ( \gamma_\lambda(e_0))\right)^{\gamma_\lambda^{-1}} = M ( 2 \lambda^{-1} \mu ( t - \lambda ) )^{\gamma_\lambda^{-1}} ,$$
where $M\left(2 \mu ( \gamma_\lambda(e_0))\right)$ is the Verma module of highest weight $2 \mu ( \gamma_\lambda(e_0))$.  In particular, $\Res_{\mathfrak{sl}_2^\Vir}^\Vir (V^{t- \lambda}_{\mu})$ is simple if and only if $2 \mu ( \gamma_\lambda(e_0)) \not\in \z_{\geq 0}$.
\end{prop}
\begin{proof}
Note that $\Vir^{t - \lambda} \cap \,  \s_2^\Vir = \mbox{span} \{e_1 - \lambda e_0, e_0 - \lambda e_{-1} \}$ and $\B_\lambda= \mbox{span} \{ e-\frac{\lambda}{2} h, \frac{1}{2} h + \lambda f \}$. Therefore, $\Vir^{t - \lambda} \cap \s_2^\Vir \cong \B_\lambda$. Using the natural embedding of $\s_2$ in $\Vir$ (where $h = 2e_0$), we consider $\mu$ as a $\s_2$-homomorphism. Then from Propositions \ref{prop:resVirMod-sl2} and \ref{prop:Vermaiso}, we have 
$$\Res_{\s_2^\Vir}^\Vir  (V^{t- \lambda}_{\mu}) \cong \Ind_{\B_\lambda}^{\s_2} ( \cc_\mu ) \cong M(\mu(\gamma_\lambda( 2 e_0 )))^{\gamma_\lambda^{-1}}.$$

To see that $2 \mu ( \gamma_\lambda ( e_0 ))= 2 \mu ( t^{-1} ( t - \lambda ))$, note that by (\ref{eqn:gammalambda}) and (\ref{eqn:sl2-embed}), 
$$\mu ( \gamma_\lambda ( 2e_0 )) = \mu ( \gamma_\lambda ( h )) = \mu ( h +  2 \lambda f) = \mu ( 2 (e_0 - \lambda e_{-1} )) = 2 \mu ( t^{-1} ( t - \lambda )).$$
Since $\mu ( t^i ( t - \lambda )) = \lambda^i \mu ( t - \lambda )$, it follows that $\mu ( \gamma_\lambda ( 2e_0 )) = 2 \lambda^{-1} \mu ( t - \lambda )$.
\end{proof}

\begin{cor}
Let $\lambda \in \cc^*$ and $\mu : \Vir^{t - \lambda} \to \cc$ be a Lie algebra homomorphism.  Then the Casimir element of $\s_2^\Vir$ is $c=-4e_{-1}e_1 +(2e_0+1)^2 \in U( \s_2^\Vir ) \subseteq U ( \Vir )$, and $c$ acts by the scalar $(2 \mu(\gamma_\lambda(e_0)) + 1)^2$ on the $\Vir$-module $V_{\mu}^{t - \lambda}$.
\end{cor}
\begin{proof}
It follows directly from the standard embedding of $\s_2$ in $\Vir$ that $c=-4e_{-1}e_1 +(2e_0+1)^2$. Since $c$ acts on the Verma module $M\left(2 \mu ( \gamma_\lambda(e_0))\right)$ by $(2 \mu ( \gamma_\lambda(e_0)) + 1)^2$, it is clear that $\gamma_\lambda (c)$ acts on $M\left(2\mu ( \gamma_\lambda(e_0))\right)^{\gamma_\lambda^{-1}}$ by $(2 \mu ( \gamma_\lambda(e_0)) + 1)^2$.  It is straightforward to compute that $\gamma_\lambda (c) = c$, so the result follows from Proposition \ref{prop:Res1}. 
\end{proof}

Note that Corollary \ref{cor:codimIntersect} and Proposition \ref{prop:Res1} hold for the polynomial $f(t) = t - \lambda$ only if $\lambda \neq 0$. In \cite{ondWies18}, we have shown that the subalgebras $\Vir^{t- \lambda} \subseteq \Vir$ for $\lambda \neq 0$ are all the codimension-one subalgebras of $\Vir$. Thus, Corollary \ref{cor:codimIntersect} gives a one-to-one correspondence between codimension-one subalgebras of $\Vir$ and codimension one subalgebras of $\s_2$ excluding $\B^\pm$; and Proposition \ref{prop:Res1} shows that this correspondence carries over to induced modules.

We note that generic $\Vir$-Verma modules are simple yet do not restrict to simple $\s_2$-modules.  In particular, the dimensions of weight spaces of a Verma module for $\Vir$ are counted by partitions, whereas the weight spaces of highest weight $\s_2$-modules are necessarily one-dimensional.   Hence the restriction of simple $\Vir$ Verma modules $M ( \zeta , \lambda )$ to $\s_2^\Vir$ is not a simple highest weight $\s_2^\Vir$-module.

\subsection{Connecting $W(\eta)$ and $X(\xi)$ to $V^{(t-\lambda_1)(t-\lambda_2)}_\mu$}

Below we consider the $\Vir$-module $V^{(t-\lambda_1)(t-\lambda_2)}_{\mu}$ for $\lambda_1, \lambda_2 \in \cc^*$ and examine the restriction of this module to $\s_2^\Vir$.  In particular,  Proposition \ref{prop:Res1} shows that each nontrivial twist of the modules  $W(\eta)$ and $X(\xi)$ lifts to a $\Vir$ module induced from a codimension $2$ $\Vir$-subalgebra.

\begin{prop} \label{prop:Res2}
Choose $\lambda_1, \lambda_2 \in \cc^*$ and $\mu: \Vir^{(t-\lambda_1)(t-\lambda_2)} \rightarrow \cc$ a nonzero Lie algebra homomorphism.   \begin{itemize}
\item[(i)] If $\lambda_1=\lambda_2:=\lambda$, then
$$\Res_{\mathfrak{sl}_2^\Vir}^{\Vir} (V^{(t-\lambda)^2}_{\mu}) \cong W\left(\mu(\gamma_\lambda(e_1))\right)^{\gamma_\lambda^{-1}}.$$

\item[(ii)] If $\lambda_1 \neq \lambda_2$, then
$$\Res_{\mathfrak{sl}_2^\Vir}^{\Vir} (V^{(t-\lambda_1)(t-\lambda_2)}_{\mu})\cong X \left( 2\mu(\gamma_{\lambda_1, \lambda_2}(e_0)) \right)^{(\gamma_{\lambda_1, \lambda_2})^{-1}}.$$
\end{itemize}
\end{prop}

\begin{proof}
If $\lambda_1=\lambda_2:=\lambda$, then 
$$\s_2^\Vir \cap \Vir^{(t-\lambda)^2} =\mbox{span} \{ e_1-2 \lambda e_0 + \lambda^2 e_{-1} \} \cong \mathfrak{n}_\lambda.$$ If $\lambda_1 \neq \lambda_2$, then 
$$\s_2^\Vir \cap \Vir^{(t-\lambda_1)(t-\lambda_2)} =\mbox{span} \{ e_1 - (\lambda_1 + \lambda_2) e_0 +\lambda_1 \lambda_2 e_{-1} \} \cong \mathfrak{h}_{\lambda_1, \lambda_2}.$$
Both results then follow from Propositions \ref{prop:twist} and \ref{prop:resVirMod-sl2}.
\end{proof}

As in Proposition \ref{prop:Res1}, we note that Proposition \ref{prop:Res2} holds only if $\lambda_1, \lambda_2 \neq 0$. In particular, this result, in conjunction with Proposition \ref{prop:quotientXmu}, shows that for a suitable choice of $\lambda_1, \lambda_2 \in \cc^*$, there is a surjective homomomorphism from $\Res_{\mathfrak{sl}_2^\Vir}^{\Vir} (V^{(t-\lambda_1)(t-\lambda_2)}_{\mu})$ to a (nontrivial) twist of one of the simple $\s_2$-modules ${\bf V} (\overline \mu_0 , \tau)$, where $\overline \mu_0$ denotes the coset $2\mu(\gamma_{\lambda_1, \lambda_2}(e_0)) + 2 \z$.

Martin and Prieto \cite{marPrie} showed that  the untwisted $\s_2$-modules $V(\overline \mu, \tau)$ can be lifted to intermediate series modules for $\Vir$. Intermediate series modules are weight modules and are, in general, simple.  Thus, there are no (non-trivial) homomorphisms between intermediate series modules and polynomial modules. Moreover, it can be shown that the twist of any polynomial $\Vir$-module is again a polynomial module. Thus, there appears to be no connection between intermediate series modules and polynomial modules, yet these modules are closely connected when restricted to $\s_2$.

As an application of the above results on restriction, we can now show that the tensor product of two twisted $\s_2$ Verma modules (or equivalently, the tensor product of two Verma modules with respect to different Borel subalgebras) is a twist of one of the modules $X(\xi)$.

\begin{cor}\label{cor:tensorVermas}
Let $\lambda_1, \lambda_2 \in \cc^*$, $\lambda_1 \neq \lambda_2$, and $\tau_1, \tau_2 \in \cc$. Then
$$M(\tau_1)^{\gamma_{\lambda_1}^{-1}} \otimes M(\tau_2)^{\gamma_{\lambda_2}^{-1}} \cong X( \tau_1 - \tau_2)^{(\gamma_{\lambda_1, \lambda_2})^{-1}}.$$
\end{cor}
\begin{proof}
From Proposition \ref{prop:Res1}, we have $M(\tau_i)^{\gamma_{\lambda_i}^{-1}} \cong \Res^{\Vir}_{\s_2^\Vir} (V^{(t-\lambda_i)}_{\mu_i})$, where $\mu_i: \Vir^{(t-\lambda_i)} \rightarrow \cc$ is a Lie algebra homomorphism such that $\tau_i =\mu_i ( \gamma_{\lambda_i} ( 2e_0) ) = \mu_i ( 2 t^{-1} ( t - \lambda_i ))$.

Corollary 6.6 of \cite{ondWies18} gives that
$V_{{\mu_1}}^{t - \lambda_1}  \otimes V_{{\mu_2}}^{t - \lambda_2} 
\cong  V_{\mu_1 + \mu_2}^{(t - \lambda_1)(t - \lambda_2)}$, and it follows that
$$
\Res^{\Vir}_{\s_2^\Vir} (  V_{{\mu_1}}^{t - \lambda_1} ) \otimes \Res^{\Vir}_{\s_2^\Vir}( V_{{\mu_2}}^{t - \lambda_2})
\cong \Res^{\Vir}_{\s_2^\Vir} \left( V_{\mu_1 +  \mu_2}^{(t - \lambda_1)(t - \lambda_2)}  \right).
$$
Using Proposition \ref{prop:Res2}, we know that 
$$ \Res^{\Vir}_{\s_2^\Vir} \left( V_{ \mu_1 +\mu_2}^{(t - \lambda_1)(t - \lambda_2)}  \right) \cong X \left(2 ( \mu_1 +  \mu_2)(\gamma_{\lambda_1, \lambda_2}(e_0))\right)^{(\gamma_{\lambda_1, \lambda_2})^{-1}}.$$
Note that
\begin{align*}
\gamma_{\lambda_1, \lambda_2}(e_0) &= \frac{1}{\lambda_2 -\lambda_1}(-e_1+(\lambda_1+ \lambda_2) e_0 - \lambda_1 \lambda_2 e_{-1}) \\
&= \frac{1}{\lambda_2 -\lambda_1} \left( -(e_1-\lambda_1 e_0) + \lambda_2 (e_0-\lambda_1 e_{-1}) \right)\\
&=  \frac{1}{\lambda_2 -\lambda_1} \left( -(e_1-\lambda_2 e_0) + \lambda_1 (e_0-\lambda_2 e_{-1}) \right).
\end{align*}
Therefore,
\begin{align*}
 \mu_1(\gamma_{\lambda_1, \lambda_2}(e_0)) &= \frac{1}{2 (\lambda_2-\lambda_1)} (-\lambda_1 \tau_1 + \lambda_2 \tau_1) = \frac{\tau_1}{2};\\
\mu_2(\gamma_{\lambda_1, \lambda_2}(e_0)) &= \frac{1}{2(\lambda_2-\lambda_1) } (-\lambda_2 \tau_2 + \lambda_1 \tau_2) = \frac{-\tau_2}{2}.
\end{align*}
Consequently $2 ( \mu_1 +  \mu_2)(\gamma_{\lambda_1, \lambda_2}(e_0)) = \tau_1 - \tau_2$, and the result follows.
\end{proof}

\subsection{The trivial subalgebra of $\s_2$ and related $\Vir$-modules}

Consider the case $f(t) = a_0 + a_1 t + a_2 t^2 + t^3$ with $a_i \in \cc$ and $a_0 \neq 0$.  In this case $\s_2^\Vir / ( \s_2^\Vir \cap \Vir^{f(t)} )$ is three-dimensional, so $\s_2^\Vir \cap \Vir^{f(t)} = \{ 0 \}$.

As a result, we see that if $V_{ \mu}^{f(t)}$ and $\mu : \Vir^{f(t)} \to \cc$ are defined in the usual way, then Proposition \ref{prop:resVirMod-sl2} implies that 
$$\Res_{\mathfrak{sl}_2^\Vir}^\Vir (V^{f(t)}_{\mu}) \cong \Ind^{\mathfrak{sl}_2^\Vir}_{\{ 0 \}}  \cc_{\mu},$$
where we regard $\cc_{\mu}$ as a one-dimensional module for the trivial Lie algebra $\{ 0 \}$.  Notice that 
$$\Ind_{\{ 0 \}}^{\s_2^\Vir} \cc_{\mu} = U( \s_2^\Vir ) \otimes_{U( \{ 0 \} )} \cc_{\mu}.$$
Since $U( \{ 0 \} ) = \cc$, it follows that $\Ind_{\{ 0 \}}^{\s_2^\Vir} \cc_{\mu} = U( \s_2^\Vir ) \otimes_\cc \cc_{\mu} \cong U( \s_2^\Vir )$, and thus
$$\Res_{\mathfrak{sl}_2^\Vir}^\Vir (V^{f(t)}_{ \mu}) \cong U( \s_2^\Vir ).$$

\subsection{Lifting $\s_2$-modules to polynomial $\Vir$-modules} \label{sec:UniqueLift}
So far, we have taken the perspective of restricting polynomial $\Vir$-modules to $\s_2$-modules. If we reverse that question, to think of the prior results in terms of lifting $\s_2$-modules to $\Vir$-modules, it is natural to ask which $\s_2$-modules can be lifted and in how many ways. Here, we reframe the prior results of the paper to address these questions.

The previous section shows that $U(\s_2)$, viewed as an $\s_2$-module, can be lifted to a $\Vir$-module in infinitely many ways, once for each choice of polynomial $f(t) = a_0 + a_1 t + a_2 t^2 + t^3$ and homomorphism $\mu : \Vir^{f(t)} \to \cc$.   We now take up the other $\s_2$-modules considered in the paper.  We begin by studying possible isomorphisms between modules of the form $M\left( \tau \right)^{\gamma_{\lambda}^{-1}}$, $W \left( \nu \right)^{\gamma_{\lambda}^{-1}}$, and $X \left( \xi  \right)^{(\gamma_{\lambda_1, \lambda_2})^{-1}}$. The lemma below establishes a key computational fact.

\begin{lem}\label{lem:eigenOperators}
Let $\tau , \nu , \xi \in \cc$, and consider $y \in \s_2$.  
\begin{itemize}
\item[(i)] Suppose $0 \neq v \in M(\tau)$ such that $yv=0$. Then $y \in \cc h \oplus \cc e$.
\item[(ii)] Suppose $0 \neq v \in W(\nu)$ such that $yv=\nu'v$ for some $\nu' \in \cc$. Then $y \in \cc e$.
\item[(iii)] Suppose $0 \neq v \in X(\xi)$ such that $yv=\xi'v$ for some $\xi' \in \cc$. Then $y \in \cc h$.
\end{itemize}
\end{lem}

\begin{proof}
 In each case, suppose $y= \alpha_1e+\alpha_2h+\alpha_3f$ for some $\alpha_i \in \cc$. The lemma statements follow from a consideration of an appropriate basis for each module. 

For example, the Verma module $M(\tau)$ has a basis $\{ f^i v^+ \mid i \ge 0 \}$, where $e v^+ = 0$ and $h v^+ = \tau v^+$. Therefore, we can write $v = \sum_{i=0}^n a_i f^i v^+$ with $a_i \in \cc$ and $a_n \neq 0$.  Then it's clear that $yv$ has a nonzero coefficient on the $f^{n+1} v^+$ term unless $\alpha_3=0$. This proves (i).

For (ii), we can use the $W(\nu)$-basis $\{f^i h^j w_{\nu} \mid i, j \geq 0\}$ (where $w_\nu$ is the unique Whittaker vector, up to scalar multiple, that generates $W(\nu)$) to argue $\alpha_3=0$.  Now with $y = \alpha_1 e + \alpha_2 h$, we can use the same argument applied to the reordered basis $\{h^i f^j w_{\nu} \mid i, j \geq 0\}$ to argue that $\alpha_2=0$.

Similar arguments can be used to prove (iii), using the fact that both $\{f^i e^j x_{\nu} \mid i, j \geq 0\}$ and $\{e^i f^j x_{\nu} \mid i, j \geq 0\}$ are bases for $X(\xi)$.
\end{proof}

Proposition \ref{prop:noIsomorphisms} below shows, for example, that $M(\tau)^{\gamma_{\lambda}^{-1}} \not\cong W(\nu)^{\gamma_{\eta}^{-1}}$; and that $W(\nu_1 )^{\gamma_{\lambda_1}^{-1}} \cong W(\nu_2 )^{\gamma_{\lambda_2}^{-1}}$ implies that $\nu_1 = \nu_2$ and $\lambda_1 = \lambda_2$.  This severely limits the possibilities for distinct polynomial $\Vir$-modules that restrict to isomorphic $\s_2$-modules.

\begin{prop}\label{prop:noIsomorphisms}
Let $V_1$ and $V_2$ be $\s_2$-modules of the form $M(\tau)^{\gamma_\lambda^{-1}}$ or $W(\nu)^{\gamma_\eta^{-1}}$ or $X(\xi)^{\gamma_{\lambda_1, \lambda_2}^{-1}}$ for allowable choices of $\tau, \nu, \xi, \lambda , \eta , \lambda_1, \lambda_2$.  Then $V_1 \cong V_2$ only if $V_1=V_2$ (i.e. $V_1$ and $V_2$ are identically constructed as twists of induced modules).
\end{prop}
\begin{proof}
The proof involves a case-by-case analysis, using Lemma \ref{lem:eigenOperators} and focusing on elements of the form $( \gamma^{-1} \circ \gamma' )(y)$ for particular choices of $y \in \s_2$ and $\s_2$-automorphisms $\gamma$ and $\gamma'$.  We present the argument for two cases below.  The remaining cases are similar.

To exemplify the argument that there are no isomorphisms between modules of different types, we will show $M(\tau)^{\gamma_\lambda^{-1}} \not\cong X(\xi)^{\gamma_{\lambda_1, \lambda_2}^{-1}}$, or equivalently that $M( \tau ) \not\cong X(\xi)^{\gamma_{\lambda_1, \lambda_2}^{-1} \circ \gamma_\lambda}$. Using Lemma \ref{lem:eigenOperators},  it's enough to check that $( \gamma_{\lambda_1, \lambda_2}^{-1} \circ \gamma_\lambda ) (h) \not\in \cc h$.  It can be shown that 
$$( \gamma_{\lambda_1, \lambda_2}^{-1} \circ \gamma_\lambda )(h)= \frac{1}{\lambda_2-\lambda_1} \left((\lambda_1+\lambda_2-2\lambda) h -2(\lambda_1-\lambda)f + 2(\lambda_2-\lambda)e\right)$$
Since $\lambda_1 \neq \lambda_2$, we must have that at least one of $\lambda_1-\lambda, \lambda_2-\lambda$ is nonzero. This shows that $( \gamma_{\lambda_1, \lambda_2}^{-1} \circ \gamma_\lambda )(h) \not\in \cc h$.

To exemplify the argument addressing isomorphisms between modules of the same type, suppose $W \left( \nu_1 \right)^{\gamma_{\lambda_1}^{-1}} \cong W \left( \nu_2 \right)^{\gamma_{\lambda_2}^{-1}}$. Then $W ( \nu_1 ) \cong W ( \nu_2 )^\gamma$, where $\gamma = \gamma_{\lambda_2}^{-1} \circ \gamma_{\lambda_1} = \gamma_{\lambda_1 - \lambda_2}$. It follows that there exists a nonzero $v \in W( \nu_2 )$ such that $\gamma ( e) v = \nu_1 v$; and thus Lemma \ref{lem:eigenOperators} implies that $\gamma (e) \in \cc e$.  Since $\gamma ( e) = e - ( \lambda_1 - \lambda_2) h - ( \lambda_1 - \lambda_2)^2 f$, it must be that $\lambda_1 = \lambda_2$, and thus $W ( \nu_1 ) \cong W( \nu_2 )$.  This implies that $\nu_1 = \nu_2$, either by direct computation or by appealing to general results on block decomposition in Section 2.3 of \cite{BatMaz}.
\end{proof}

Consider the $\Vir$-modules  $V^{f_1(t)}_{\mu_1}$ and $V^{f_2(t)}_{\mu_2}$, where $f_i(t)$ is a polynomial of the form $t-\lambda$, $(t-\lambda)^2$ or $(t-\lambda_1)(t-\lambda_2)$ for some $\lambda, \lambda_i \in \cc^*$.  From Proposition 5.12 of \cite{ondWies18}, we know that if $V^{f_1(t)}_{\mu_1} \cong V^{f_2(t)}_{\mu_2}$, then $f_1(t)=f_2(t)$ and $\mu_1= \mu_2$.  Below we build on Proposition \ref{prop:noIsomorphisms} to show that if $\Res_{\mathfrak{sl}_2^\Vir}^\Vir  (V^{f_1 (t)}_{\mu_1}) \cong \Res_{\mathfrak{sl}_2^\Vir}^\Vir  (V^{f_2 (t)}_{\mu_2})$, then $f_1(t)=f_2(t)$ and $\mu_1$ and $\mu_2$ are closely connected.


\begin{lem}\label{lem:muOptions}
Let $\lambda, \lambda_1, \lambda_2 \in \cc^*$ with $\lambda_1 \neq \lambda_2$. 
\begin{itemize}
\item[(i)] For a fixed $\tau \in \cc$, there is a unique Lie algebra homomorphism $\mu: \Vir^{(t-\lambda)} \rightarrow \cc$ such that $\mu(\gamma_{\lambda}(2e_0))=\tau$.
\item[(ii)] For a fixed $\nu \in \cc$, there is a one-to-one correspondence between $\{ \mu: \Vir^{(t-\lambda)^2} \rightarrow \cc \mid \mu(\gamma_{\lambda}(e_1))=\nu \}$ and $\cc$.
\item[(iii)] For a fixed $\xi \in \cc$, 
 there is a one-to-one correspondence between \\$\{ \mu: \Vir^{(t-\lambda_1)(t-\lambda_2)} \rightarrow \cc \mid \mu(\gamma_{\lambda_1 , \lambda_2} (2e_0))=\xi \}$ and $\cc$.
\end{itemize}
\end{lem}

\begin{proof}
As shown in Lemma 3.2 of \cite{ondWies18}, if $f(t) = \prod_{i} ( t - \lambda_i )^{n_i}$, then a linear map $\mu: \Vir^{f(t)} \rightarrow \cc$ is a Lie algebra homomorphism precisely when both $\mu(z)=0$ and there exist polynomials $p_1, \ldots,p_k$ with $\deg(p_i) < n_i$ such that $\mu(t^j f) =p_1(j)\lambda_1^j+ \cdots + p_k(j) \lambda_k^j$.

For the first case, suppose $f(t)=t-\lambda$. Then $\mu(t^i f(t))= p \lambda^i$ for some constant $p \in \cc$. It's clear that there is only one choice of $p$ such that $\tau = \mu (\gamma_{\lambda}(2e_0))=2 \lambda^{-1} \mu(t-\lambda)=2 \lambda^{-1} p$, namely $p = \frac{\tau \lambda}{2}$.

For the second case, suppose $f(t)=(t-\lambda)^2$. Then $\mu(t^j f)= (p_0 +p_1 j)\lambda^j$ for some constants $p_0 , p_1 \in \cc$. In particular, $\nu= \mu (\gamma_{\lambda} (e_1))= \mu (t^{-1}(t-\lambda)^2)=\frac{p_0 - p_1}{\lambda}$.  It follows that $\mu ( \gamma_\lambda (e_1) ) = \nu$ if and only if $p_0 = \nu \lambda + p_1$.  In other words, such a $\mu$ is completely determined by $p_1 \in \cc$ and has the form $\mu(t^j f(t))= ( \nu \lambda + p_1 + p_1 j ) \lambda^j$.

For the third case, suppose $f(t)=(t-\lambda_1)(t-\lambda_2)$.  Then $\mu(t^j f(t))= p_1\lambda_1^j+p_2 \lambda_2^j$ for some constants $p_1, p_2 \in \cc$. In particular, 
$$\xi=\mu(\gamma_{\lambda_1 , \lambda_2}(2e_0)) = \mu ( \gamma_{\lambda_1 , \lambda_2} (h) ) 
= \frac{2}{\lambda_1 - \lambda_2} \mu(t^{-1}(t-\lambda_1)(t-\lambda_2))$$
if and only if $\frac{( \lambda_1 - \lambda_2 ) \xi}{2} = \frac{p_1}{\lambda_1} + \frac{p_2}{\lambda_2}$.  This implies that for every $p_1 \in \cc$, there is a unique $p_2 \in \cc$ such that $\xi = \mu(\gamma_{\lambda_1 , \lambda_2}(2e_0))$.  In other words, such a $\mu$ is completely determined by $p_1 \in \cc$.
\end{proof}

The result below summarizes some of the general connections between induced $\s_2$-modules and polynomial modules for $\Vir$.

\begin{thm} \label{thm:almostUnique}
Suppose $N$ is an $\s_2$-module that is induced from a simple module for a  nontrivial, proper subalgebra $\mathfrak a \subseteq \s_2$ such that $\mathfrak a \not\subseteq \B^+, \B^-$. 
\begin{itemize}
\item[(i)] If $\mathfrak a$ has codimension one, then there is a unique polynomial module $V^{f (t)}_\mu$ such that ${\rm Res}^{\Vir}_{\s_2} \, ( V^{f (t)}_\mu ) \cong N$.  
\item[(ii)] If $\mathfrak a$ has codimension two, then there is a unique polynomial $f(t)$ along with infinitely many $\mu$, in one-to-one correspondence with $\cc$, such that ${\rm Res}^{\Vir}_{\s_2} \, ( V^{f (t)}_\mu ) \cong N$. 
\end{itemize}
 In either case, the degree of $f(t)$ is the same as the codimension of $\mathfrak a$ in $\s_2$. 
\end{thm}

\begin{proof}
Lemmas \ref{lem:onedim} and \ref{lem:2dimSubs} describe all one-dimensional and two-dimensional subalgebras of $\s_2$. 
It follows from Lemma \ref{lem:inducedtwist} and Propositions \ref{prop:Res1}, \ref{prop:Res2} and \ref{prop:noIsomorphisms} and Lemma \ref{lem:muOptions} that the claims of the theorem apply to all subalgebras $\mathfrak a$ except the following:
\begin{itemize}
\item $\mathfrak n_\lambda$ if $\lambda=0$ (in which case $\mathfrak n_0= \mathfrak n^+ = \cc e$);
\item $\mathfrak n^-$;
\item $\mathfrak h_\lambda = \cc (h+2\lambda f)$;
\item $\mathfrak h_{\lambda_1, \lambda_2}$ if $\lambda_1=0$ (in which case $\mathfrak h_{\lambda_1, 0}=\cc(\frac{2}{\lambda_1} e-h)$) or $\lambda_2=0$ (in which case $\mathfrak h_{0, \lambda_2}=\cc(-\frac{2}{\lambda_2} e+h)$). 
\end{itemize}
It's clear that all subalgebras on this list are subalgebras of $\B^+$ or $\B^-$. Moreover, it follows from (\ref{eqn:gammalambda}) and (\ref{eqn:tau}) and Lemmas \ref{lem:onedim} and \ref{lem:2dimSubs} that all other nontrivial, proper subalgebras of $\s_2$ are not subalgebras of $\B^+$ or $\B^-$.
\end{proof}

\end{document}